\documentclass[11pt,a4paper]{article}

\usepackage{inputenc}
\usepackage{amsmath}
\usepackage{bm}
\usepackage{bbold}
\usepackage{amsthm}
\usepackage{enumerate}

\usepackage{pseudocode}
\usepackage[hyphens]{url}

\usepackage{hyperref}
\usepackage{breakurl}

\title{Complete solution of tropical vector inequalities using matrix sparsification}

\author{N. Krivulin\thanks{Faculty of Mathematics and Mechanics, Saint Petersburg State University, 28 Universitetsky Ave., St.~Petersburg, 198504, Russia, 
nkk@math.spbu.ru.}
\thanks{This work was supported in part by the Russian Foundation for Basic Research (grant No. 20-010-00145).}}

\date{}

\newtheorem{theorem}{Theorem}
\newtheorem{lemma}[theorem]{Lemma}
\newtheorem{corollary}[theorem]{Corollary}

\theoremstyle{definition}
\newtheorem{example}{Example}

\begin{document}

\maketitle

\begin{abstract}
We examine the problem of finding all solutions of two-sided vector inequalities given in the tropical algebra setting, where the unknown vector multiplied by known matrices appears on both sides of the inequality. We offer a solution that uses sparse matrices to simplify the problem and to construct a family of solution sets, each defined by a sparse matrix obtained from one of the given matrices by setting some of its entries to zero. All solutions are then combined to present the result in a parametric form in terms of a matrix whose columns form a complete system of generators for the solution. We describe the computational technique proposed to solve the problem, remark on its computational complexity and illustrate this technique with numerical examples.
\\

\\

\end{abstract}

\section{Introduction}

The problem of solving two-sided vector inequalities in the tropical algebra setting (where the unknown vector multiplied by known matrices appears on both sides of the inequality) occurs in a variety of contexts, from geometry of tropical polyhedral cones \cite{Zimmermann1977General,Allamigeon2011Number,Gaubert2009Tropical} to mean payoff games \cite{Akian2012Tropical,Gaubert2012Tropical}. In its general form, the two-sided inequality is represented as $\bm{A}\bm{x}\leq\bm{B}\bm{x}$, where $\bm{A}$ and $\bm{B}$ are given matrices, $\bm{x}$ is the unknown vector, and the matrix-vector multiplication is interpreted in terms of a tropical semifield (a semiring with idempotent addition and invertible multiplication).

The problem of solving the two-sided inequality is closely related to the solution of the two-sided equation $\bm{A}\bm{x}=\bm{B}\bm{x}$, which is equivalent to two opposite two-sided inequalities. Both problems have no known polynomial-time solution, and hence are considered hard to solve. The available solutions for these inequality and equation comprise various algorithmic procedures and computational schemes developed during the last decades in a range of works, including \cite{Butkovic1984Elimination,Cuninghamegreen2003Equation,Butkovic2006Strongly,Wagneur2009Tropical,Lorenzo2011Algorithm,Sergeev2011Basic,Gaubert2012Tropical,Jones2019Twosided}. Existing solutions of the two-sided inequality are based on elimination techniques that successively simplify the problem by eliminating scalar inequalities \cite{Butkovic1984Elimination}, discrete algorithms that examine subsets of indices of rows and columns of the matrices involved \cite{Wagneur2009Tropical,Sergeev2011Basic}, reduction procedures that transform the problem into a set of problems of lower dimension \cite{Lorenzo2011Algorithm}, and other approaches.

Although the number of algorithmic methods continues to grow, the derivation of direct, analytical solutions of two-sided inequalities and equations is still a challenging problem. A complete solution in closed form is known only for the inequality $\bm{A}\bm{x}\leq\bm{x}$, which is a special case of the two-sided inequality where the matrix on the right-hand side is reduced as $\bm{B}=\bm{I}$ \cite{Krivulin2015Multidimensional,Krivulin2015Extremal} (see also \cite{Sergeev2009Multiorder,Elsner2010Maxalgebra,Sergeev2011Basic}). Application of this solution in combination with algorithmic techniques seems to be promising to handle the inequality in the general case.

As an example of this combined approach, one can consider the solution based on matrix sparsification techniques in \cite{Krivulin2017Algebraic,Krivulin2018Complete}. Specifically in \cite{Krivulin2018Complete} a complete solution is derived for the inequality with a reduced left-hand side in the form $\bm{x}\leq\bm{B}\bm{x}$, which replaces this inequality by a collection of inequalities with the reduced right-hand side, each solvable in explicit form, and then combines the solutions into one.

In this paper we extend the above solution of the inequality $\bm{x}\leq\bm{B}\bm{x}$ to handle the general two-sided inequality $\bm{A}\bm{x}\leq\bm{B}\bm{x}$. We follow an approach which uses sparse matrices to construct a family of solution subsets, each defined by a sparse matrix obtained from the matrix $\bm{B}$ by setting some of its entries to zero. All solutions are then combined to present the result in a parametric form in terms of a matrix whose columns compose a complete system of generators for the solution.

Since the brute-force production of the subsets results in exponential growth of the number of subsets to examine, we propose a backtracking algorithm that produces the subsets in an economic way to avoid excess computation.

The paper is organized as follows. Section~\ref{S-PDNR} presents a short introduction into the tropical algebra to provide an overview of the basic facts, symbols and results used in the subsequent sections. In Section~\ref{S-TSI} we formulate the problem and make some observations on the solutions. The main result is included in Section~\ref{S-CSTSI}, which provides a complete solution of the two-sided inequality. In Section~\ref{S-CIS} we discuss the computational implementation of the solution, and describe a procedure of generating solution sets. Section~\ref{S-NE} offers numerical examples to illustrate the results.

\section{Preliminary definitions, notation and results}
\label{S-PDNR}

We start with a brief overview of basic definitions, notation and preliminary results of tropical algebra, which underlie the solutions presented below. For further details one can consult, e.g., \cite{Baccelli1993Synchronization,Cuninghamegreen1994Minimax,Kolokoltsov1997Idempotent,Golan2003Semirings,Heidergott2006Maxplus,Gondran2008Graphs,Butkovic2010Maxlinear}.

\subsection{Idempotent semifield}

Consider a nonempty set $\mathbb{X}$ that is closed under addition $\oplus$ and multiplication $\otimes$, and has zero $\mathbb{0}$ and one $\mathbb{1}$ as the neutral elements of the operations $\oplus$ and $\otimes$. It is assumed that $(\mathbb{X},\oplus,\mathbb{0})$ is a commutative idempotent monoid, $(\mathbb{X}\setminus\{\mathbb{0}\},\oplus,\mathbb{1})$ is an Abelian group, and multiplication distributes over addition. The system $(\mathbb{X},\oplus,\otimes,\mathbb{0},\mathbb{1})$ is referred to as the idempotent semifield.

Integer powers specify repeated multiplication: $\mathbb{0}^{p}=\mathbb{0}$, $x^{0}=\mathbb{1}$, $x^{p}=x\otimes x^{p-1}$ and $x^{-p}=(x^{-1})^{p}$, where $x^{-1}$ is the inverse of $x$, for any nonzero $x\in\mathbb{X}$ and natural $p$. The integer powers are assumed to extend to powers with rational exponents. In what follows the multiplication symbol $\otimes$ is omitted to save writing.

The idempotent addition induces a partial order on $\mathbb{X}$ such that $x\leq y$ if and only if $x\oplus y=y$. With respect to this order, the addition possesses the extremal properties (the majority law of addition) in the form of the inequalities $x\leq x\oplus y$ and $y\leq x\oplus y$ satisfied for any $x,y\in\mathbb{X}$. Furthermore, addition and multiplication are isotone which means that the inequality $x\leq y$ yields $x\oplus z\leq y\oplus z$ and $xz\leq yz$ for any $z$. The inversion is antitone: $x\leq y$ results in $x^{-1}\geq y^{-1}$ for $x,y\ne\mathbb{0}$. Finally, the inequality $x\oplus y\leq z$ is equivalent to the pair of inequalities $x\leq z$ and $y\leq z$. The partial order is assumed to extend to a total order to make $\mathbb{X}$ linearly ordered.

An example of idempotent semifield under consideration is the real semifield $\mathbb{R}_{\max,+}=(\mathbb{R}\cup\{-\infty\},\max,+,-\infty,0)$ which is often called $(\max,+)$-algebra. In this semifield we have the operations defined as $\oplus=\max$ and $\otimes=+$, and the neutral elements as $\mathbb{0}=-\infty$ and $\mathbb{1}=0$. Furthermore, the inverse $x^{-1}$ of $x\in\mathbb{R}$ coincides with the opposite number $-x$ in standard arithmetic. The power $x^{y}$ corresponds to the arithmetic product $xy$ which is defined for all $x,y\in\mathbb{R}$. Finally, the order induced by idempotent addition agrees with the natural linear order on $\mathbb{R}$.

\subsection{Matrices and vectors}

The set of matrices with $m$ rows and $n$ columns over $\mathbb{X}$ is denoted by $\mathbb{X}^{m\times n}$. A matrix with all entries equal to $\mathbb{0}$ is the zero matrix denoted by $\bm{0}$. A matrix is called \textit{row-regular} if it has no rows with all entries equal to $\mathbb{0}$. 

The addition and multiplication of conforming matrices, and multiplication of matrices by scalars follow the standard rules where the arithmetic addition and multiplication are replaced by the scalar operations $\oplus$ and $\otimes$.

For any nonzero matrix $\bm{A}=(a_{ij})\in\mathbb{X}^{m\times n}$, the multiplicative inverse transpose (or the conjugate \cite{Cuninghamegreen1976Projections,Cuninghamegreen1994Minimax}) is the matrix $\bm{A}^{-}=(a_{ij}^{-})\in\mathbb{X}^{n\times m}$ where $a_{ij}^{-}=a_{ji}^{-1}$ if $a_{ji}\ne\mathbb{0}$, and $a_{ij}^{-}=\mathbb{0}$ otherwise. 

Properties of scalar operations with respect to the order relations are extended to the matrix operations where the relations are interpreted componentwise.

A square matrix with $\mathbb{1}$ on the diagonal and $\mathbb{0}$ elsewhere is the identity matrix denoted by $\bm{I}$. The non-negative integer powers of a square nonzero matrix $\bm{A}$ are defined in the usual way: $\bm{A}^{0}=\bm{I}$ and $\bm{A}^{p}=\bm{A}\bm{A}^{p-1}$ for any natural $p$.

The trace of a matrix $\bm{A}=(a_{ij})\in\mathbb{X}^{n\times n}$ is given by $\mathop\mathrm{tr}\bm{A}=a_{11}\oplus\cdots\oplus a_{nn}$. The trace possesses usual properties, specifically it is invariant under cyclic permutations: $\mathop\mathrm{tr}(\bm{A}\bm{B})=\mathop\mathrm{tr}(\bm{B}\bm{A})$ for any matrices $\bm{A}$ and $\bm{B}$ of compatible sizes.

Any matrix that consists of one column (row) forms a column (row) vector. A vector with all elements equal to $\mathbb{0}$ is the zero vector. For the sake of simplicity, the zero vector is denoted by the same symbol as the zero matrix $\bm{0}$. Any vector $\bm{x}$ without zero elements is called \textit{regular}, which can be expressed as $\bm{x}>\bm{0}$.

All vectors below are considered column vectors unless transposed. The set of column vectors over $\mathbb{X}$ with $n$ elements is denoted by $\mathbb{X}^{n}$.  

A row-regular matrix that has exactly one nonzero entry in each row is called \textit{strictly row-monomial} \cite{Clifford1961Algebraic}. The next statement describes properties of these matrices.
\begin{lemma}
If a matrix $\bm{A}$ is strictly row-monomial, then $\bm{A}^{-}\bm{A}\leq\bm{I}\leq\bm{A}\bm{A}^{-}$.
\end{lemma}
\begin{proof}
To verify the first inequality $\bm{A}^{-}\bm{A}\leq\bm{I}$, we show that it holds for each pair of rows of the matrices on both sides. Consider the matrix $\bm{A}^{-}$, and note that each row with all entries equal to $\mathbb{0}$, if it exists in $\bm{A}^{-}$, yields the same zero row in $\bm{A}^{-}\bm{A}$ on the left-hand side, which is trivially less than any row on the right.

Next, we check that every row of $\bm{A}^{-}$ with nonzero entries produces a pair of coinciding rows on both sides of the first inequality. The multiplication of such row by the column of the same number in $\bm{A}$ sets the diagonal entry of this row in $\bm{A}^{-}\bm{A}$ to $\mathbb{1}$ to make it equal to the corresponding entry in $\bm{I}$. Moreover, this row must have the other entries equal to $\mathbb{0}$ as in $\bm{I}$, since otherwise the matrix $\bm{A}$ is to have more than one nonzero entries in a row, which contradicts that $\bm{A}$ is row-monomial.

Considering that for any row-regular matrix $\bm{A}$, all diagonal entries in $\bm{A}\bm{A}^{-}$ are equal to $\mathbb{1}$, we conclude that the second inequality $\bm{I}\leq\bm{A}\bm{A}^{-}$ holds as well.
\end{proof}

Finally, note that if a matrix $\bm{A}$ is strictly row-monomial, then the inequality $\bm{A}\bm{x}\geq\bm{y}$ where $\bm{x}$ and $\bm{y}$ are vectors is equivalent to $\bm{x}\geq\bm{A}^{-}\bm{y}$. Indeed, we can multiply the first inequality by $\bm{A}^{-}$ on the left to obtain $\bm{x}\geq\bm{A}^{-}\bm{A}\bm{x}\geq\bm{A}^{-}\bm{y}$, which yields the second. At the same time, the multiplication of the second inequality by $\bm{A}$ on the left produces the first inequality as $\bm{A}\bm{x}\geq\bm{A}\bm{A}^{-}\bm{y}\geq\bm{y}$.

\subsection{Linear dependence}

A vector $\bm{b}\in\mathbb{X}^{m}$ is \textit{linearly dependent} on vectors $\bm{a}_{1},\ldots,\bm{a}_{n}\in\mathbb{X}^{m}$ if there exist scalars $x_{1},\ldots,x_{n}\in\mathbb{X}$ such that $\bm{b}=x_{1}\bm{a}_{1}\oplus\cdots\oplus x_{n}\bm{a}_{n}$. A vector $\bm{b}$ is collinear with $\bm{a}$ if $\bm{b}=x\bm{a}$ for some scalar $x$.

To test the linear dependence of vectors, various formal criteria \cite{Cuninghamegreen1979Minimax,Cuninghamegreen2004Bases,Butkovic2010Maxlinear} are used based on existence conditions for solutions of the equation $\bm{A}\bm{x}=\bm{b}$ where $\bm{A}$ is a matrix built from the vectors $\bm{a}_{1},\ldots,\bm{a}_{n}$. Specifically, the following result \cite{Krivulin2006Solution} (see also \cite{Cuninghamegreen1994Minimax}) offers a simple criterion, which requires no more than $O(mn)$ operations.
\begin{lemma}
\label{L-AbAb1}
A vector $\bm{b}$ is linearly dependent on vectors $\bm{a}_{1},\ldots,\bm{a}_{n}$ if and only if the condition $(\bm{A}(\bm{b}^{-}\bm{A})^{-})^{-}\bm{b}=\mathbb{1}$ holds with matrix $\bm{A}=(\bm{a}_{1},\ldots,\bm{a}_{n})$. 
\end{lemma}

A system of vectors $\bm{a}_{1},\ldots,\bm{a}_{n}$ is linearly dependent if at least one vector is linearly dependent on others. Two systems of vectors are equivalent if each vector of one system is linearly dependent on vectors of the other system.

Consider a system $\bm{a}_{1},\ldots,\bm{a}_{n}$ that may have linearly dependent vectors. To construct an equivalent linearly independent system, we can use a procedure that successively reduces the system until it becomes linearly independent (see, e.g., \cite{Cuninghamegreen2004Bases,Butkovic2010Maxlinear}). The procedure applies the criterion provided by Lemma~\ref{L-AbAb1} to examine the vectors one by one. It removes a vector if it is linearly dependent on others, or leaves the vector in the system otherwise. As one can see, the procedure yields a linearly independent system equivalent to the original one with no more than $O(mn^{2})$ operations.

\subsection{Solution of vector inequality}

We now present a complete solution to a two-sided vector inequality of a special form. Given a matrix $\bm{A}\in\mathbb{X}^{n\times n}$, consider the problem to find regular vectors $\bm{x}\in\mathbb{X}^{n}$ to satisfy the inequality
\begin{equation}
\bm{A}\bm{x}
\leq
\bm{x}.
\label{I-Axleqx}
\end{equation}

The problem is examined in different contexts under various assumptions in some publications, including \cite{Sergeev2009Multiorder,Elsner2010Maxalgebra,Sergeev2011Basic}, where similar solutions based on the concept of the Kleene star operator (the Kleene closure) are given.

To describe a solution in explicit form, we introduce a function that maps any square matrix $\bm{A}\in\mathbb{X}^{n\times n}$ onto the scalar $\mathop\mathrm{Tr}(\bm{A})=\mathop\mathrm{tr}\bm{A}\oplus\cdots\oplus\mathop\mathrm{tr}\bm{A}^{n}$.

Consider the operator (the Kleene star) which takes the matrix $\bm{A}$ to the series
\begin{equation*}
\bm{A}^{\ast}
=
\bm{I}
\oplus
\bm{A}
\oplus
\bm{A}^{2}
\oplus
\cdots
\end{equation*}

Suppose the condition $\mathop\mathrm{Tr}(\bm{A})\leq\mathbb{1}$ holds, which implies each cyclic product of entries in $\bm{A}$, including the diagonal entries, is less than or equal to $\mathbb{1}$. Then, the series converges (see, e.g., \cite{Carre1971Algebra,Cuninghamegreen1979Minimax} and also \cite{Yoeli1961Note}) so that the Kleene star becomes
\begin{equation*}
\bm{A}^{\ast}
=
\bm{I}\oplus\bm{A}
\oplus\cdots\oplus
\bm{A}^{n-1}.
\end{equation*}
    
As a consequence, the inequality $\bm{A}^{\ast}\geq\bm{A}^{k}$ is then valid for all integers $k\geq0$.

The next result offers a complete solution in a parametric form \cite{Krivulin2015Multidimensional,Krivulin2015Extremal}.
\begin{theorem}
\label{T-Axleqx}
For any matrix $\bm{A}$ the following statements hold.
\begin{enumerate}
\item
If $\mathop\mathrm{Tr}(\bm{A})\leq\mathbb{1}$, then all regular solutions to \eqref{I-Axleqx} are given in parametric form by $\bm{x}=\bm{A}^{\ast}\bm{u}$ where $\bm{u}$ is any regular vector.
\item
If $\mathop\mathrm{Tr}(\bm{A})>\mathbb{1}$, then there is only the trivial solution $\bm{x}=\bm{0}$.
\end{enumerate}
\end{theorem}

We note that the calculation of both the function $\mathop\mathrm{Tr}(\bm{A})$ and the matrix $\bm{A}^{\ast}$ in the theorem requires at most $O(n^{4})$ operations if direct matrix computations are used, and $O(n^{3})$ operations with an application of the Floyd-Warshall algorithm.

\section{Two-sided inequality}
\label{S-TSI}

We are now in a position to formulate the problem of interest. Given matrices $\bm{A},\bm{B}\in\mathbb{X}^{m\times n}$, we need to find regular vectors $\bm{x}\in\mathbb{X}^{n}$ that satisfy the inequality
\begin{equation}
\bm{A}\bm{x}
\leq
\bm{B}\bm{x}.
\label{I-AxleqBx}
\end{equation}

It is not difficult to verify that the solution set of inequality \eqref{I-AxleqBx} is closed under vector addition and scalar multiplication. Indeed, let $\bm{x}$ and $\bm{y}$ be vectors such that the inequalities $\bm{A}\bm{x}\leq\bm{B}\bm{x}$ and $\bm{A}\bm{y}\leq\bm{B}\bm{y}$ hold, and consider any vector $\bm{z}=\alpha\bm{x}\oplus\beta\bm{y}$ where $\alpha$ and $\beta$ are scalars. Then, we have $\bm{A}\bm{z}=\alpha\bm{A}\bm{x}\oplus\beta\bm{A}\bm{y}\leq\alpha\bm{B}\bm{x}\oplus\beta\bm{B}\bm{y}=\bm{B}\bm{z}$, which means that the vector $\bm{z}$ is a solution of inequality \eqref{I-AxleqBx} as well.

Without loss of generality, we may assume that both matrices $\bm{A}$ and $\bm{B}$ are row-regular. Otherwise, if the matrix $\bm{A}$ has a zero row, say row $i$, then the corresponding scalar inequality $a_{i1}x_{1}\oplus\cdots\oplus a_{in}x_{n}\leq b_{i1}x_{1}\oplus\cdots\oplus b_{in}x_{n}$ trivially holds, and thus this inequality can be removed, whereas row $i$ eliminated from both matrices. Let the matrix $\bm{A}$ be row-regular and suppose the matrix $\bm{B}$ has a zero row $i$, which leads to the inequality $a_{i1}x_{1}\oplus\cdots\oplus a_{in}x_{n}\leq\mathbb{0}$. This inequality holds only if each unknown $x_{j}$ with $a_{ij}\ne\mathbb{0}$ is set to zero, which results in the non-regular solution of no interest.

Under additional assumptions some solutions of inequality \eqref{I-AxleqBx} can be directly obtained in explicit form. As an example, consider the next result.
\begin{lemma}
\label{L-AxleqBx}
Let $\bm{A}$ and $\bm{B}$ be row-regular matrices such that $\mathop\mathrm{Tr}(\bm{B}^{-}\bm{A})\leq\mathbb{1}$. Then, inequality \eqref{I-AxleqBx} has solutions given by
\begin{equation*}
\bm{x}
=
(\bm{B}^{-}\bm{A})^{\ast}\bm{u},
\qquad
\bm{u}>\bm{0}.
\end{equation*}
\end{lemma}
\begin{proof}
By Theorem~\ref{T-Axleqx}, the condition $\mathop\mathrm{Tr}(\bm{B}^{-}\bm{A})\leq\mathbb{1}$ is equivalent to the existence of regular solutions $\bm{x}$ of the inequality $\bm{B}^{-}\bm{A}\bm{x}\leq\bm{x}$. Moreover, all regular solutions of this inequality are given by $\bm{x}=(\bm{B}^{-}\bm{A})^{\ast}\bm{u}$ where $\bm{u}$ is a regular vector of parameters.

Since the matrix $\bm{B}$ is row-regular, and thus $\bm{I}\leq\bm{B}\bm{B}^{-}$, the multiplication of the inequality $\bm{B}^{-}\bm{A}\bm{x}\leq\bm{x}$ by $\bm{B}$ on the left yields $\bm{A}\bm{x}\leq\bm{B}\bm{B}^{-}\bm{A}\bm{x}\leq\bm{B}\bm{x}$, which shows that all solutions of this inequality satisfy inequality \eqref{I-AxleqBx} as well.
\end{proof}

\section{Complete solution using sparse matrices}
\label{S-CSTSI}

We now derive a complete solution of inequality \eqref{I-AxleqBx} by applying a matrix sparsification technique to represent all solutions as a family of solution sets in parametric form. Each member of the family is described by a generating matrix calculated with a strictly row-monomial matrix obtained from the matrix $\bm{B}$ on the right-hand side of \eqref{I-AxleqBx}. Then, we combine all solutions by using a single generating matrix.

To handle inequality \eqref{I-AxleqBx}, in a similar way as in \cite{Butkovic1984Elimination,Wagneur2009Tropical} we first set to $\mathbb{0}$ each entry of the matrices $\bm{A}$ and $\bm{B}$ that do not affect the set of regular solutions. The next statement introduces the sparsified matrices obtained as a result.

\begin{lemma}
\label{L-aij-bij}
Let $\bm{A}=(a_{ij})$ and $\bm{B}=(b_{ij})$ be row-regular matrices. Define the sparsified matrices $\widehat{\bm{A}}=(\widehat{a}_{ij})$ and $\widehat{\bm{B}}=(\widehat{b}_{ij})$ with the entries
\begin{equation}
\widehat{a}_{ij}
=
\begin{cases}
a_{ij},
&
\text{\upshape if $a_{ij}>b_{ij}$};
\\
\mathbb{0},
&
\text{\upshape otherwise};
\end{cases}
\qquad
\widehat{b}_{ij}
=
\begin{cases}
b_{ij},
&
\text{\upshape if $b_{ij}\geq a_{ij}$};
\\
\mathbb{0},
&
\text{\upshape otherwise}.
\end{cases}
\label{E-aij-bij}
\end{equation}

Then, replacing the matrix $\bm{A}$ by $\widehat{\bm{A}}$ and $\bm{B}$ by $\widehat{\bm{B}}$ does not change the regular solutions of inequality \eqref{I-AxleqBx}.
\end{lemma}
\begin{proof}
For each $i=1,\ldots,m$, consider all nonzero solutions $x_{1},\ldots,x_{n}\in\mathbb{X}$ of the inequality which corresponds to row $i$ in the matrices $\bm{A}$ and $\bm{B}$, and takes the form
\begin{equation} 
a_{i1}x_{1}\oplus\cdots\oplus
a_{in}x_{n}\leq
b_{i1}x_{1}\oplus\cdots\oplus
b_{in}x_{n}.
\label{I-ai1xiainxnleqbi1x1binxn}
\end{equation} 

Suppose that the condition $a_{ij}\leq b_{ij}$ holds for some $j=1,\ldots,n$. Then, the inequality $a_{ij}x_{j}\leq b_{ij}x_{j}\leq b_{i1}x_{1}\oplus\cdots\oplus b_{in}x_{n}$ is valid for all $x_{j}\in\mathbb{X}$, which shows that the term $a_{ij}x_{j}$ cannot be greater than the right-hand side of inequality \eqref{I-ai1xiainxnleqbi1x1binxn}. Observing that this term cannot violate \eqref{I-ai1xiainxnleqbi1x1binxn}, it can be eliminated by setting $a_{ij}=\mathbb{0}$.  

If the condition $a_{ij}>b_{ij}$ is satisfied, then the inequality $a_{ij}x_{j}>b_{ij}x_{j}$ is valid for all $x_{j}\ne\mathbb{0}$. Since the term $b_{ij}x_{j}$ does not contribute to the right-hand side of \eqref{I-ai1xiainxnleqbi1x1binxn}, we can set $b_{ij}=\mathbb{0}$ without affecting all regular solutions. 
\end{proof}

As it is easy to see, the above replacement procedure requires $O(mn)$ operations.

In what follows the matrices $\widehat{\bm{A}}$ and $\widehat{\bm{B}}$ obtained from $\bm{A}$ and $\bm{B}$ according to \eqref{E-aij-bij} are referred to as \textit{refined matrices} of the inequality or simply as refined matrices.

\subsection{Solution using matrix sparsification}

To derive a family of subsets that describe all solutions of inequality \eqref{I-AxleqBx}, we state and prove the next theorem.  
\begin{theorem}
\label{T-AxleqBx}
Let $\bm{A}$ and $\bm{B}$ be refined row-regular matrices, and $\bm{G}$ be a strictly row-monomial matrix obtained from $\bm{B}$ by fixing one nonzero entry in each row while setting the others to $\mathbb{0}$. Denote by $\mathcal{G}$ the set of the matrices $\bm{G}$ which satisfy the condition $\mathop\mathrm{tr}\bm{H}^{n}\leq\mathbb{1}$ where $\bm{H}=\bm{G}^{-}(\bm{A}\oplus\bm{B})$.

Then, all regular solutions of inequality \eqref{I-AxleqBx} are given by the conditions
\begin{equation}
\bm{x}
=
(\bm{I}\oplus\bm{H}^{n-1})\bm{u},
\qquad
\bm{H}=\bm{G}^{-}(\bm{A}\oplus\bm{B}),
\qquad
\bm{G}
\in
\mathcal{G},
\qquad
\bm{u}
>
\bm{0}.
\label{E-xeqIHn1u-HeqGAB-GinG-uge0}
\end{equation}
\end{theorem}

\begin{proof}
Let us show that any regular solution of inequality \eqref{I-AxleqBx} is given by the conditions at \eqref{E-xeqIHn1u-HeqGAB-GinG-uge0} and vice versa. We take a regular solution $\bm{x}=(x_{j})$ of \eqref{I-AxleqBx} with matrices $\bm{A}=(a_{ij})$ and $\bm{B}=(b_{ij})$, and examine, for certain $p$, the scalar inequality 
\begin{equation}
a_{p1}x_{1}\oplus\cdots\oplus a_{pn}x_{n}
\leq
b_{p1}x_{1}\oplus\cdots\oplus b_{pn}x_{n}.
\label{I-ap1x1apnxnleqbp1x1bpnxn}
\end{equation}

Our purpose is to reduce, under appropriate conditions, this inequality to a simpler inequality with one term on the right, which presents a key component of the solution approach. Suppose that inequality \eqref{I-ap1x1apnxnleqbp1x1bpnxn} holds for some $x_{1},\ldots,x_{n}$ and consider the sum on the right-hand side. Since the order defined by the relation $\leq$ is linear, we can pick out a term, say $b_{pq}x_{q}$, that is maximal among all terms, and thereby produces the value of the sum. Then, we can replace \eqref{I-ap1x1apnxnleqbp1x1bpnxn} by two inequalities $b_{pq}x_{q}\geq b_{p1}x_{1}\oplus\cdots\oplus b_{pn}x_{n}$ and $b_{pq}x_{q}\geq a_{p1}x_{1}\oplus\cdots\oplus a_{pn}x_{n}$ where $b_{pq}>\mathbb{0}$. Finally, we combine these inequalities into an equivalent inequality that is given by 
\begin{equation}
b_{pq}x_{q}
\geq
(a_{p1}\oplus b_{p1})x_{1}\oplus\cdots\oplus(a_{pn}\oplus b_{pn})x_{n}.
\label{I-bpqxqgeqap1bp1x1apnbpnxn}
\end{equation}

Further assume that we select maximum terms in all scalar inequalities in \eqref{I-AxleqBx} and then substitute an inequality in the form of \eqref{I-bpqxqgeqap1bp1x1apnbpnxn} for each scalar inequality. Let $\bm{G}$ be a strictly row-monomial matrix that is formed from $\bm{B}$ by fixing the entry which corresponds to the maximum term in each row, while setting the other entries to $\mathbb{0}$. By using matrix $\bm{G}$, the inequalities obtained for each row are combined into the vector inequality $\bm{G}\bm{x}\geq(\bm{A}\oplus\bm{B})\bm{x}$. Since the matrix $\bm{G}$ is strictly row-monomial,  this vector inequality is equivalent to the inequality $\bm{x}\geq\bm{G}^{-}(\bm{A}\oplus\bm{B})\bm{x}$, which takes the form $\bm{x}\geq\bm{H}\bm{x}$ with the notation $\bm{H}=\bm{G}^{-}(\bm{A}\oplus\bm{B})$.

By assumption the inequality $\bm{x}\geq\bm{H}\bm{x}$ has a regular solution $\bm{x}$. Then, it follows from Theorem~\ref{T-Axleqx} that the condition $\mathop\mathrm{Tr}(\bm{H})\leq\mathbb{1}$ holds, whereas all regular solutions of the inequality are given by $\bm{x}=\bm{H}^{\ast}\bm{u}$ with a regular vector $\bm{u}$.

Consider now a vector $\bm{x}=\bm{H}^{\ast}\bm{u}$ where $\bm{u}$ is a regular vector, and $\bm{H}=\bm{G}^{-}(\bm{A}\oplus\bm{B})$ with a strictly row-monomial matrix $\bm{G}$ such that $\mathop\mathrm{Tr}(\bm{H})\leq\mathbb{1}$. To verify that $\bm{x}$ satisfies inequality \eqref{I-AxleqBx}, we note that $\bm{H}^{\ast}\geq\bm{H}^{k}$ for all $k\geq0$ since $\mathop\mathrm{Tr}(\bm{H})\leq\mathbb{1}$, and hence $\bm{H}^{\ast}\geq\bm{H}\oplus\cdots\oplus\bm{H}^{n}=\bm{H}\bm{H}^{\ast}$. Moreover, we see that $\bm{H}=\bm{G}^{-}(\bm{A}\oplus\bm{B})\geq\bm{G}^{-}\bm{A}$ and $\bm{B}\bm{G}^{-}\geq\bm{G}\bm{G}^{-}\geq\bm{I}$ as $\bm{B}\geq\bm{G}$. Using these inequalities yields
\begin{equation*}
\bm{B}\bm{H}^{\ast}
\geq
\bm{B}\bm{H}\bm{H}^{\ast}
\geq
\bm{B}\bm{G}^{-}\bm{A}\bm{H}^{\ast}
\geq
\bm{A}\bm{H}^{\ast}. 
\end{equation*}

Therefore, we see that $\bm{B}\bm{x}=\bm{B}\bm{H}^{\ast}\bm{u}\geq\bm{A}\bm{H}^{\ast}\bm{u}=\bm{A}\bm{x}$, and thus $\bm{x}$ satisfies \eqref{I-AxleqBx}.

It remains to verify that $\bm{H}^{\ast}=\bm{I}\oplus \bm{H}^{n-1}$ and $\mathop\mathrm{Tr}(\bm{H})=\mathop\mathrm{tr}\bm{H}^{n}$ so as to represent the solution as in the statement of the theorem. Since $\bm{H}=\bm{G}^{-}(\bm{A}\oplus\bm{B})\geq\bm{G}^{-}\bm{B}$, we have $\bm{H}^{2}=(\bm{G}^{-}\bm{A}\oplus\bm{G}^{-}\bm{B})^{2}\geq\bm{G}^{-}\bm{B}(\bm{G}^{-}\bm{A}\oplus\bm{G}^{-}\bm{B})\geq\bm{G}^{-}\bm{B}\bm{G}^{-}(\bm{A}\oplus\bm{B})$. By observing that $\bm{B}\bm{G}^{-}\geq\bm{I}$, we further obtain
\begin{equation*}
\bm{H}^{2}
\geq
\bm{G}^{-}\bm{B}\bm{G}^{-}
(\bm{A}\oplus\bm{B})
\geq
\bm{G}^{-}(\bm{A}\oplus\bm{B})
=
\bm{H},
\end{equation*}
and then conclude that $\bm{H}^{k+1}\geq\bm{H}^{k}$ for all $k\geq1$. As a result, the equalities $\bm{H}^{\ast}=\bm{I}\oplus\bm{H}\oplus\cdots\oplus\bm{H}^{n-1}=\bm{I}\oplus\bm{H}^{n-1}$ and $\mathop\mathrm{Tr}(\bm{H})=\mathop\mathrm{tr}\bm{H}\oplus\cdots\oplus\mathop\mathrm{tr}\bm{H}^{n}=\mathop\mathrm{tr}\bm{H}^{n}$ are valid, which completes the proof.
\end{proof}

Let us note that to find one solution subset according to Theorem~\ref{T-AxleqBx}, we need to solve inequality \eqref{I-Axleqx} with the matrix $\bm{H}=\bm{G}^{-}(\bm{A}\oplus\bm{B})$. Since $\bm{H}$ is a square matrix of order $n$, one can obtain the solution in $O(n^{4})$ operations by direct matrix multiplications or in $O(n^{3})$ with the Floyd-Warshall algorithm. By observing that $\bm{H}^{k}=\bm{G}^{-}\bm{F}^{k-1}(\bm{A}\oplus\bm{B})$ holds for all $k\geq1$, where $\bm{F}=(\bm{A}\oplus\bm{B})\bm{G}^{-}$ is a matrix of order $m$, we refine these estimates of complexity by replacing $n$ by $\min(m,n)$.

In the simplest case when the refined matrix $\bm{B}$ has no more than one nonzero entry in each column, which yields only one strictly row-monomial matrix $\bm{G}$, the above estimates determine the overall complexity of complete solution. As the number of nonzero entries in the matrix $\bm{B}$ rises, the number of matrices $\bm{G}$ that can be obtained from $\bm{B}$ increases very rapidly, and becomes exponentially large in the worst case.

As a worst case for the brute-force generation of the row-monomial matrices, one can consider the case of a matrix $\bm{B}$ without zero entries, which formally yields $n^{m}$ matrices $\bm{G}$. To overcome the problem of increasing complexity, we propose below a backtracking procedure intended to reduce the number of matrices $\bm{G}$ to examine.

\subsection{Closed-form representation of solution}

The next result shows how to represent all solutions in a compact parametric form using a single generating matrix.

\begin{corollary}
Under the conditions and notations of Theorem~\ref{T-AxleqBx}, denote by $\bm{S}$ the matrix whose columns form the maximal independent system of columns in the matrices $\bm{I}\oplus\bm{H}^{n-1}=\bm{I}\oplus(\bm{G}^{-}(\bm{A}\oplus\bm{B}))^{n-1}$ for all $\bm{G}\in\mathcal{G}$. 

Then, all regular solutions of inequality \eqref{I-AxleqBx} are given in parametric form by
\begin{equation*}
\bm{x}
=
\bm{S}\bm{v},
\qquad
\bm{v}
>
\bm{0}.
\end{equation*}
\end{corollary}
\begin{proof}
By Theorem~\ref{T-AxleqBx}, the solution set is the union of sets, each generated by the columns of the matrix $\bm{I}\oplus\bm{H}^{n-1}=\bm{I}\oplus(\bm{G}^{-}(\bm{A}\oplus\bm{B}))^{n-1}$ for all $\bm{G}\in\mathcal{G}$. Observing that the solution set is closed under vector addition and scalar multiplication, we conclude that this union is the linear span of all columns in the generating matrices.

Furthermore, we reduce the set of columns by eliminating those which are linearly dependent on others, and thus can be deleted without affecting the entire linear span. With the matrix $\bm{S}$ formed from the reduced set of columns, all solutions are given in parametric form by $\bm{x}=\bm{S}\bm{v}$ where $\bm{v}$ is any regular vector of appropriate size. 
\end{proof}

If only one matrix $\bm{G}$ is available, the derivation of the generating matrix $\bm{S}$ reduces to checking the linear dependence of columns of one matrix $\bm{I}\oplus\bm{H}^{n-1}$, which requires $O(n^{3})$ operations. It is not difficult to see that in the worst case, the computational complexity increases as square of the number of matrices $\bm{G}$ obtained.

\section{Computational implementation of solution}
\label{S-CIS}

To derive a complete solution of two-sided inequality \eqref{I-AxleqBx}, we offer a solution procedure that involves: (i) preliminary refinement of the matrices, (ii) generation of the solution sets, and (iii) derivation of the matrix which generates all solutions.

\subsection{Refinement of matrices}

We begin the procedure with the refinement of the matrices according to Lemma~\ref{L-aij-bij}. Suppose the inequality, after refinement, gets zero rows in the matrix $\bm{A}$ or $\bm{B}$. Then, as it is shown before, one can reduce the inequality by deleting some rows in $\bm{A}$ and $\bm{B}$ or conclude that there is no regular solution.

Provided that both matrices $\bm{A}$ and $\bm{B}$ upon refinement are row-regular, the procedure passes to the next step of generating the family of solution sets.

\subsection{Generation of solution sets}

Consider the solution offered to inequality \eqref{I-AxleqBx} by Theorem~\ref{I-AxleqBx} in the form of a family of solution sets and note that each member of the family involves a strictly row-monomial matrix $\bm{G}$ to calculate the corresponding generating matrix $\bm{I}\oplus\bm{H}^{n-1}$ from the matrix $\bm{H}=\bm{G}^{-}(\bm{A}\oplus\bm{B})$. The matrices $\bm{G}$ are successively obtained from the matrix $\bm{B}$ by setting to $\mathbb{0}$ all but one of the entries in each row of $\bm{B}$. Since the number of the strictly row-monomial matrices may be excessively large, we propose a backtracking procedure that aims at rejecting in advance those matrices which cannot serve to provide a solution.  

The procedure consecutively checks rows $i=1,\ldots,n$ of the matrix $\bm{B}$ to find and fix, over all $j=1,\ldots,n$, a nonzero entry $b_{ij}$, while setting the other entries to $\mathbb{0}$. The selection of a nonzero entry $b_{pq}$ in row $p$ implies that the term $b_{pq}x_{q}$ is taken maximal over all $q$, which establishes relations between $x_{q}$ and $x_{j}$ with $j\ne q$. We exploit these relations to modify entries in the remaining rows by setting them to $\mathbb{0}$ provided that these entries cannot affect the corresponding scalar inequalities in \eqref{I-AxleqBx}. One step of the procedure is completed when a nonzero entry is fixed in the last row, which makes a new strictly row-monomial matrix $\bm{G}$ fully defined.

A new step of the procedure is to take the next nonzero entry in the current row if such an entry exists. Otherwise, the procedure has to go back to the previous row to cancel the last selection of nonzero entry and roll back all modifications made to the matrix in accordance with this selection. Next, the procedure fixes a new nonzero entry in this row, if it exists, or continues back to the previous rows until an unexplored nonzero entry is found. On selection of a new entry, the procedure continues forward to modify and fix nonzero entries in the next rows.

The procedure is repeated until no more nonzero entries can be selected in the first row. A description of the procedure in recursive form is given in Algorithm~\ref{A-GenerateSparseMatrices}.

\begin{table}[h]
\renewcommand{\COMMENT}[1]{\mbox{\bfseries comment:}\ \mbox{#1}}
\begin{pseudocode}{GenerateSparseMatrices}{\bm{B},\mathcal{G}}\label{A-GenerateSparseMatrices}
\PROCEDURE{Backtrack}{\bm{B},p,q}
\COMMENT{Sparsify rows $i\geq p$ in the matrix $\bm{G}=(b_{ij})$}
\\
\IF p\leq m \THEN
\BEGIN
	\COMMENT{Verify whether $b_{pq}$ can be fixed in row $p$}
	\\
	\IF b_{pq}\ne\mathbb{0} \THEN
	\BEGIN
		\COMMENT{Copy $\bm{B}$ into the matrix $\bm{B}^{\prime}=(b_{ij}^{\prime})$}
		\\
		\bm{B}^{\prime} \GETS \bm{B}
		\\
		\COMMENT{Sparsify row $p$ in $\bm{B}^{\prime}$ with $b_{pq}^{\prime}$ fixed}
		\\
		\FOREACH j\ne q \ADO b_{pj}^{\prime} \GETS \mathbb{0}
			\\	
			\COMMENT{Sparsify rows $i>p$ in $\bm{B}^{\prime}$}
			\\
			\FOR i \GETS p+1 \TO m \DO
			\BEGIN
				\text{flag} \GETS \TRUE
				\\
				\FOREACH j\ne q \DO
					\BEGIN
					\IF b_{iq}b_{pq}^{-1}(a_{pj}\oplus b_{pj})<a_{ij} \THEN
						\BEGIN
							\text{flag} \GETS \FALSE
							\\
							\BREAK
						\END
					\END	
					\\
				\IF \text{flag} \THEN
				\FOREACH j\ne q \ADO b_{ij}^{\prime} \GETS \mathbb{0}
				\ELSE
				\BEGIN
					\FOREACH j\ne q \DO
						\BEGIN
							\IF b_{iq}b_{pq}^{-1}(a_{pj}\oplus b_{pj})\geq b_{ij} \THEN
							b_{ij}^{\prime} \GETS \mathbb{0}
						\END
						\\
				\END
			\END
			\\
			\IF p=m \THEN
			\BEGIN
				\COMMENT{Store $\bm{B}^{\prime}$ if completed}
				\\
				\mathcal{G} \GETS \mathcal{G}\cup\{\bm{B}^{\prime}\}
			\END
			\ELSE
			\BEGIN
				\COMMENT{Apply recursion otherwise}
				\\
				\FOR j \GETS 1 \TO n \DO
				\CALL{Backtrack}{\bm{B}^{\prime},p+1,j} 
			\END
		\END
	\END
	\ELSE \RETURN{}
\ENDPROCEDURE
\MAIN
\COMMENT{Generate the set $\mathcal{G}$ of sparse matrices from the matrix $\bm{B}$}
\\
\GLOBAL{m,n,\bm{A},\mathcal{G}=\emptyset}
\\
\FOR j \GETS 1 \TO n \DO
	\CALL{Backtrack}{\bm{B},1,j}
\ENDMAIN	
\end{pseudocode}
\end{table}

To describe the row modification routine in the procedure in more detail, suppose there are nonzero entries fixed in rows $i=1,\ldots,p-1$, and the procedure now selects an entry $b_{pq}$ in row $p$. Since this selection implies that $b_{pq}x_{q}$ is assumed to be the maximum term with $b_{pq}>\mathbb{0}$ in the right-hand side of inequality \eqref{I-ap1x1apnxnleqbp1x1bpnxn}, it follows from \eqref{I-bpqxqgeqap1bp1x1apnbpnxn} that the inequality $x_{q}\geq b_{pq}^{-1}(a_{pj}\oplus b_{pj})x_{j}$ is satisfied for all $j=1,\ldots,n$.

We use two criteria to test whether an entry in the matrix $\bm{B}$ can be set to $\mathbb{0}$ in the course of the building of a row-monomial matrix. Consider inequality \eqref{I-ai1xiainxnleqbi1x1binxn} for $i=p+1,\ldots,n$. Provided the condition $b_{iq}b_{pq}^{-1}(a_{pj}\oplus b_{pj})\geq a_{ij}$ holds for all $j$, the term $b_{iq}x_{q}$ alone makes this inequality true because $b_{iq}x_{q}\geq b_{iq}b_{pq}^{-1}(a_{pj}\oplus b_{pj})x_{j}\geq a_{ij}x_{j}$. Observing that the other terms do not contribute to the inequality, the entries $b_{ij}$ can be set to $\mathbb{0}$ for all $j\ne q$ without affecting the solution set under construction.

If the above condition is not valid, verify the condition $b_{iq}b_{pq}^{-1}(a_{pj}\oplus b_{pj})\geq b_{ij}$ for every $j\ne q$. Suppose that the last condition holds for some $j$, and therefore $b_{iq}x_{q}\geq b_{iq}b_{pq}^{-1}(a_{pj}\oplus b_{pj})x_{j}\geq b_{ij}x_{j}$. Since the term $b_{ij}x_{j}$ is now dominated by $b_{iq}x_{q}$, it does not affect the right-hand side of \eqref{I-ai1xiainxnleqbi1x1binxn}, which allows us to set $b_{ij}=\mathbb{0}$.  

Consider the worst case for the brute-force generation of strictly row-monomial matrices, which occurs when the matrix $\bm{B}$ has no zero entries. It is not difficult to see that an application of the first criterion results in only $n$ matrices $\bm{G}$, each formed from $\bm{B}$ by fixing the entries of one column and replacing the other entries by $\mathbb{0}$.

\subsection{Derivation of solution matrix}

The solution matrix $\bm{S}$ is formed by combining all columns of the matrices which generate the members of the solution family. To eliminate linear dependent columns, the procedure examines each new generating matrix as it becomes available. A column of this matrix is accepted to extend the matrix $\bm{S}$ if it is linearly independent of columns in $\bm{S}$ or rejected otherwise.

\section{Numerical examples}
\label{S-NE}

To illustrate the computational technique involved in the solution procedure, we present example problems to solve inequality \eqref{I-AxleqBx} in terms of the $\mathbb{R}_{\max,+}$ semifield.

\begin{example}
\label{X-1}
Suppose that the matrices on the left- and right-hand sides of the two-sided inequality are respectively given by
\begin{equation*}
\bm{A}_{0}
=
\left(
\begin{array}{rrr}
0 & 2 & 3
\\
\mathbb{0} & -1 & 3
\\
3 & 2 & -1
\end{array}
\right),
\qquad
\bm{B}_{0}
=
\left(
\begin{array}{rrr}
2 & -1 & 2
\\
1 & 0 & 2
\\
-1 & 3 & 1
\end{array}
\right),
\end{equation*}
where we use the symbol $\mathbb{0}=-\infty$ to save writing.

To solve the inequality, we first replace the matrices $\bm{A}_{0}$ and $\bm{B}_{0}$ by the refined matrices $\bm{A}$ and $\bm{B}$, and calculate the sum of the refined matrices to obtain
\begin{equation*}
\bm{A}
=
\left(
\begin{array}{ccc}
\mathbb{0} & 2 & 3
\\
\mathbb{0} & \mathbb{0} & 3
\\
3 & \mathbb{0} & \mathbb{0}
\end{array}
\right),
\qquad
\bm{B}
=
\left(
\begin{array}{ccc}
2 & \mathbb{0} & \mathbb{0}
\\
1 & 0 & \mathbb{0}
\\
\mathbb{0} & 3 & 1
\end{array}
\right).
\qquad
\bm{A}\oplus\bm{B}
=
\left(
\begin{array}{ccc}
2 & 2 & 3
\\
1 & 0 & 3
\\
3 & 3 & 1
\end{array}
\right).
\end{equation*}

We start deriving the strictly row-monomial matrices $\bm{G}$ by fixing the nonzero entry $b_{11}$ in the first row of the matrix $\bm{B}$. Next, we apply two criteria to check if any of the nonzero entries in the next rows can be replaced by zero.

Observing that $b_{21}b_{11}^{-1}(a_{13}\oplus b_{13})=2<a_{23}=3$, we see that the first criterion does not allow setting $b_{22}$ to $\mathbb{0}$. At the same time, we have $b_{21}b_{11}^{-1}(a_{12}\oplus b_{12})=1>b_{22}=0$, which means that $b_{22}$ can be set to $\mathbb{0}$ according to the second criterion.

Finally, two nonzero entries in the third row yield two matrices
\begin{equation*}
\bm{G}_{1}
=
\left(
\begin{array}{ccc}
2 & \mathbb{0} & \mathbb{0}
\\
1 & \mathbb{0} & \mathbb{0}
\\
\mathbb{0} & 3 & \mathbb{0}
\end{array}
\right),
\qquad
\bm{G}_{2}
=
\left(
\begin{array}{ccc}
2 & \mathbb{0} & \mathbb{0}
\\
1 & \mathbb{0} & \mathbb{0}
\\
\mathbb{0} & \mathbb{0} & 1
\end{array}
\right).
\end{equation*}  

Furthermore, we form the multiplicative inverse transposes
\begin{equation*}
\bm{G}_{1}^{-}
=
\left(
\begin{array}{rrr}
-2 & -1 & \mathbb{0}
\\
\mathbb{0} & \mathbb{0} & -3
\\
\mathbb{0} & \mathbb{0} & \mathbb{0}
\end{array}
\right),
\qquad
\bm{G}_{2}^{-}
=
\left(
\begin{array}{rrr}
-2 & -1 & \mathbb{0}
\\
\mathbb{0} & \mathbb{0} & \mathbb{0}
\\
\mathbb{0} & \mathbb{0} & -1
\end{array}
\right),
\end{equation*}
and then calculate the matrices
\begin{equation*}
\bm{H}_{1}
=
\bm{G}_{1}^{-}(\bm{A}\oplus\bm{B})
=
\left(
\begin{array}{rrr}
0 & 0 & 2
\\
0 & 0 & -2
\\
\mathbb{0} & \mathbb{0} & \mathbb{0}
\end{array}
\right),
\qquad
\bm{H}_{2}
=
\bm{G}_{2}^{-}(\bm{A}\oplus\bm{B})
=
\left(
\begin{array}{ccc}
0 & 0 & 2
\\
\mathbb{0} & \mathbb{0} & \mathbb{0}
\\
2 & 2 & 0
\end{array}
\right).
\end{equation*}

Next, we evaluate the second and third powers of the matrix $\bm{H}_{1}$ to obtain
\begin{equation*}
\bm{H}_{1}^{2}
=
\bm{H}_{1}^{3}
=
\left(
\begin{array}{ccc}
0 & 0 & 2
\\
0 & 0 & 2
\\
\mathbb{0} & \mathbb{0} & \mathbb{0}
\end{array}
\right),
\qquad
\mathop\mathrm{tr}\bm{H}_{1}^{3}
=
0
=
\mathbb{1},
\end{equation*}
which means that the matrix $\bm{H}_{1}$ satisfies the conditions of Theorem~\ref{I-AxleqBx}.

In contrast, the matrix $\bm{H}_{2}$ does not satisfy the conditions because
\begin{equation*}
\bm{H}_{2}^{2}
=
\left(
\begin{array}{ccc}
4 & 4 & 2
\\
\mathbb{0} & \mathbb{0} & \mathbb{0}
\\
2 & 2 & 4
\end{array}
\right),
\qquad
\bm{H}_{2}^{3}
=
\left(
\begin{array}{ccc}
4 & 4 & 6
\\
\mathbb{0} & \mathbb{0} & \mathbb{0}
\\
6 & 6 & 4
\end{array}
\right),
\qquad
\mathop\mathrm{tr}\bm{H}_{2}^{3}
=
4
>
\mathbb{1}.
\end{equation*}

As a result, we take the matrix $\bm{H}_{1}$ to form the generating matrix
\begin{equation*}
\bm{I}\oplus\bm{H}_{1}^{2}
=
\left(
\begin{array}{ccc}
0 & 0 & 2
\\
0 & 0 & 2
\\
\mathbb{0} & \mathbb{0} & 0
\end{array}
\right).
\end{equation*}

Since the first two columns in this matrix coincide, we drop one of them to represent all regular solutions of the two-sided inequality as
\begin{equation*}
\bm{x}
=
\bm{S}\bm{v},
\qquad
\bm{S}
=
\left(
\begin{array}{cc}
0 & 2
\\
0 & 2
\\
\mathbb{0} & 0
\end{array}
\right),
\qquad
\bm{v}
>
\bm{0}.
\end{equation*}

In terms of conventional algebra, the solution takes the parametric form
\begin{equation*} 
x_{1}
=
x_{2}
=
\max(v_{1},v_{2}+2),
\qquad
x_{3}
=
v_{2},
\qquad
v_{1},v_{2}
\in\mathbb{R},
\end{equation*} 
or the equivalent compact form
\begin{equation*} 
x_{1}
=
x_{2}
\geq
x_{3}+2.
\end{equation*} 

Note that we can verify the obtained result by solving the problem from scratch. In the usual setting, inequality \eqref{I-AxleqBx} corresponds to the system
\begin{align*}
\max\{x_{2}+2,x_{3}+3\}
&\leq
x_{1}+2,
\\
x_{3}+3
&\leq
\max\{x_{1}+1,x_{2}\},
\\
x_{1}+3
&\leq
\max\{x_{2}+3,x_{3}+1\}.
\end{align*}

The first inequality is equivalent to the pair of inequalities $x_{2}+2\leq x_{1}+2$ and $x_{3}+3\leq x_{1}+2$, which yield $x_{2}\leq x_{1}$ and $x_{3}+1\leq x_{1}$. Since $x_{2}\leq x_{1}$, the right-hand side of the second inequality becomes $\max\{x_{1}+1,x_{2}\}=x_{1}+1$, whereas the inequality itself reduces to $x_{3}+2\leq x_{1}$, which replaces the inequality $x_{3}+1\leq x_{1}$.

In the third inequality, the condition $x_{2}\leq x_{1}$ makes the term $x_{2}+3$ under the $\max$ operator less than the left-hand side of this inequality. If the third inequality holds, then this term cannot contribute to the right-hand side, and thus can be removed to put the third inequality in the form $x_{1}+3\leq x_{2}+3$ or equivalently, in $x_{1}\leq x_{2}$. By combining the last inequality with the inequalities $x_{2}\leq x_{1}$ and $x_{3}+2\leq x_{1}$, we come back to the solution $x_{1}=x_{2}\geq x_{3}+2$.
\end{example}

\begin{example}
Consider the inequality examined in \cite{Lorenzo2011Algorithm,Sergeev2011Basic} with the refined matrices 
\begin{equation*}
\bm{A}
=
\left(
\begin{array}{ccccccc}
\mathbb{0} & \mathbb{0} & \mathbb{0} & 0 & 4 & 2 & 6
\\
\mathbb{0} & 5 & 6 & \mathbb{0} & \mathbb{0} & \mathbb{0} & 2
\end{array}
\right),
\qquad
\bm{B}
=
\left(
\begin{array}{ccccccc}
0 & 1 & 5 & \mathbb{0} & \mathbb{0} & \mathbb{0} & \mathbb{0}
\\
3 & \mathbb{0} & \mathbb{0} & 0 & 2 & 4 & \mathbb{0}
\end{array}
\right).
\end{equation*}

To solve the problem, we first construct a set of strictly row-monomial matrices $\bm{G}$ by selecting appropriate nonzero entries in the matrix $\bm{B}$ to fix and setting the other entries to $\mathbb{0}$. Then, we test each matrix $\bm{G}$ to remove those which do not satisfy the condition given by Theorem~\ref{T-AxleqBx} for inclusion of the matrix into the set $\mathcal{G}$ of proper matrices. Finally, we use the matrices $\bm{G}\in\mathcal{G}$ to calculate the generating matrices, and combine the columns in these matrices to produce a single generating matrix.

In the first row of the matrix $\bm{B}$, we fix the entry $b_{11}=0$ and set the other nonzero entries to $\mathbb{0}$. Then, we verify whether nonzero entries other than $b_{21}$ in the second row can be replaced by $\mathbb{0}$. We begin with the first criterion which requires the condition $b_{21}b_{11}^{-1}(a_{1j}\oplus b_{1j})\geq a_{2j}$ to hold for all $j>1$. With $j=2$, we have $b_{21}b_{11}^{-1}(a_{12}\oplus b_{12})=4$ and $a_{22}=5$, which makes this condition unsatisfied. 

Next, we apply the second criterion to check the condition $b_{21}b_{11}^{-1}(a_{1j}\oplus b_{1j})\geq b_{2j}$. Since $b_{21}b_{11}^{-1}(a_{14}\oplus b_{14})=3\geq b_{24}=0$, we set $b_{24}=\mathbb{0}$ according to this criterion. In the same way, we put $b_{25}=b_{26}=\mathbb{0}$, which leads to a single matrix
\begin{equation*}
\bm{G}_{1}
=
\left(
\begin{array}{ccccccc}
0 & \mathbb{0} & \mathbb{0} & \mathbb{0} & \mathbb{0} & \mathbb{0} & \mathbb{0}
\\
3 & \mathbb{0} & \mathbb{0} & \mathbb{0} & \mathbb{0} & \mathbb{0} & \mathbb{0}
\end{array}
\right).
\end{equation*}

Let us fix the next entry $b_{12}=1$. Since $b_{22}=\mathbb{0}$, we have $b_{22}b_{12}^{-1}(a_{1j}\oplus b_{1j})=\mathbb{0}$ for all $j$, and thus both conditions of the first and second criteria cannot hold for nonzero entries of the second row in $\bm{B}$ to allow setting them to $\mathbb{0}$.

In this case, we have four matrices, each corresponding to a nonzero entry in the second row of $\bm{B}$, given by
\begin{gather*}
\bm{G}_{2}
=
\left(
\begin{array}{ccccccc}
\mathbb{0} & 1 & \mathbb{0} & \mathbb{0} & \mathbb{0} & \mathbb{0} & \mathbb{0}
\\
3 & \mathbb{0} & \mathbb{0} & \mathbb{0} & \mathbb{0} & \mathbb{0} & \mathbb{0}
\end{array}
\right),
\qquad
\bm{G}_{3}
=
\left(
\begin{array}{ccccccc}
\mathbb{0} & 1 & \mathbb{0} & \mathbb{0} & \mathbb{0} & \mathbb{0} & \mathbb{0}
\\
\mathbb{0} & \mathbb{0} & \mathbb{0} & 0 & \mathbb{0} & \mathbb{0} & \mathbb{0}
\end{array}
\right),
\\
\bm{G}_{4}
=
\left(
\begin{array}{ccccccc}
\mathbb{0} & 1 & \mathbb{0} & \mathbb{0} & \mathbb{0} & \mathbb{0} & \mathbb{0}
\\
\mathbb{0} & \mathbb{0} & \mathbb{0} & \mathbb{0} & 2 & \mathbb{0} & \mathbb{0}
\end{array}
\right),
\qquad
\bm{G}_{5}
=
\left(
\begin{array}{ccccccc}
\mathbb{0} & 1 & \mathbb{0} & \mathbb{0} & \mathbb{0} & \mathbb{0} & \mathbb{0}
\\
\mathbb{0} & \mathbb{0} & \mathbb{0} & \mathbb{0} & \mathbb{0} & 4 & \mathbb{0}
\end{array}
\right).
\end{gather*}

For the same reason, the selection of the entry $b_{13}=5$ does not reduce the number of nonzero entries in the second row of $\bm{B}$ and yields another four matrices
\begin{gather*}
\bm{G}_{6}
=
\left(
\begin{array}{ccccccc}
\mathbb{0} & \mathbb{0} & 5 & \mathbb{0} & \mathbb{0} & \mathbb{0} & \mathbb{0}
\\
3 & \mathbb{0} & \mathbb{0} & \mathbb{0} & \mathbb{0} & \mathbb{0} & \mathbb{0}
\end{array}
\right),
\qquad
\bm{G}_{7}
=
\left(
\begin{array}{ccccccc}
\mathbb{0} & \mathbb{0} & 5 & \mathbb{0} & \mathbb{0} & \mathbb{0} & \mathbb{0}
\\
\mathbb{0} & \mathbb{0} & \mathbb{0} & 0 & \mathbb{0} & \mathbb{0} & \mathbb{0}
\end{array}
\right),
\\
\bm{G}_{8}
=
\left(
\begin{array}{ccccccc}
\mathbb{0} & \mathbb{0} & 5 & \mathbb{0} & \mathbb{0} & \mathbb{0} & \mathbb{0}
\\
\mathbb{0} & \mathbb{0} & \mathbb{0} & \mathbb{0} & 2 & \mathbb{0} & \mathbb{0}
\end{array}
\right),
\qquad
\bm{G}_{9}
=
\left(
\begin{array}{ccccccc}
\mathbb{0} & \mathbb{0} & 5 & \mathbb{0} & \mathbb{0} & \mathbb{0} & \mathbb{0}
\\
\mathbb{0} & \mathbb{0} & \mathbb{0} & \mathbb{0} & \mathbb{0} & 4 & \mathbb{0}
\end{array}
\right).
\end{gather*}

We now form the matrices $\bm{H}_{i}=\bm{G}_{i}^{-}(\bm{A}\oplus\bm{B})$ and examine their traces to decide whether the matrix $\bm{G}_{i}$ is accepted or rejected for each $i=1,\ldots,9$. Since the trace is invariant under cyclic permutations, we replace the traces for $\bm{H}_{i}$ by those of the matrices $\bm{F}_{i}=(\bm{A}\oplus\bm{B})\bm{G}_{i}^{-}$ which have a lower order to simplify calculations.


First, we take the matrix $\bm{G}_{1}$ and calculate
\begin{equation*}
\bm{G}_{1}^{-}
=
\left(
\begin{array}{cr}
0 & -3
\\
\mathbb{0} & \mathbb{0}
\\
\mathbb{0} & \mathbb{0}
\\
\mathbb{0} & \mathbb{0}
\\
\mathbb{0} & \mathbb{0}
\\
\mathbb{0} & \mathbb{0}
\\
\mathbb{0} & \mathbb{0}
\end{array}
\right),
\qquad
\bm{A}\oplus\bm{B}
=
\left(
\begin{array}{ccccccc}
0 & 1 & 5 & 0 & 4 & 2 & 6
\\
3 & 5 & 6 & 0 & 2 & 4 & 2
\end{array}
\right).
\end{equation*}

Then, we successively obtain
\begin{equation*}
\bm{F}_{1}
=
(\bm{A}\oplus\bm{B})\bm{G}_{1}^{-}
=
\left(
\begin{array}{cr}
0 & -3
\\
3 & 0
\end{array}
\right),
\qquad
\bm{F}_{1}^{k}
=
\bm{F}_{1},
\quad
k\geq1;
\qquad
\mathop\mathrm{tr}\bm{F}_{1}
=
0.
\end{equation*} 

Observing that $\mathop\mathrm{tr}\bm{H}_{1}^{7}=\mathop\mathrm{tr}\bm{F}_{1}^{7}=0=\mathbb{1}$, we retain the matrix $\bm{G}_{1}$ to form the generating matrix
\begin{equation*}
\bm{S}_{1}
=
\bm{I}\oplus\bm{H}_{1}^{6}
=
\bm{I}
\oplus
\bm{G}_{1}^{-}\bm{F}_{1}^{5}(\bm{A}\oplus\bm{B})
=
\left(
\begin{array}{ccccccc}
0 & 2 & 5 & 0 & 4 & 2 & 6
\\
\mathbb{0} & 0 & \mathbb{0} & \mathbb{0} & \mathbb{0} & \mathbb{0} & \mathbb{0}
\\
\mathbb{0} & \mathbb{0} & 0 & \mathbb{0} & \mathbb{0} & \mathbb{0} & \mathbb{0}
\\
\mathbb{0} & \mathbb{0} & \mathbb{0} & 0 & \mathbb{0} & \mathbb{0} & \mathbb{0}
\\
\mathbb{0} & \mathbb{0} & \mathbb{0} & \mathbb{0} & 0 & \mathbb{0} & \mathbb{0}
\\
\mathbb{0} & \mathbb{0} & \mathbb{0} & \mathbb{0} & \mathbb{0} & 0 & \mathbb{0}
\\
\mathbb{0} & \mathbb{0} & \mathbb{0} & \mathbb{0} & \mathbb{0} & \mathbb{0} & 0
\end{array}
\right).
\end{equation*} 

Next, we examine the matrix $\bm{G}_{2}$ to obtain
\begin{equation*}
\bm{F}_{2}
=
(\bm{A}\oplus\bm{B})\bm{G}_{2}^{-}
=
\left(
\begin{array}{cr}
0 & -3
\\
4 & 0
\end{array}
\right),
\qquad
\bm{F}_{2}^{2}
=
\left(
\begin{array}{cr}
1 & -3
\\
4 & 1
\end{array}
\right),
\qquad
\mathop\mathrm{tr}\bm{F}_{2}^{2}
=
1.
\end{equation*}

Since $\mathop\mathrm{tr}\bm{H}_{2}^{7}=\mathop\mathrm{tr}\bm{F}_{2}^{7}\geq\mathop\mathrm{tr}\bm{F}_{2}^{2}=1>0=\mathbb{1}$, the matrix $\bm{G}_{2}$ is rejected. 

In the same way, we find
\begin{align*}
\bm{F}_{3}
&=
\left(
\begin{array}{cc}
0 & 0
\\
4 & 0
\end{array}
\right),
&
\bm{F}_{3}^{2}
&=
\left(
\begin{array}{cc}
4 & 0
\\
4 & 4
\end{array}
\right),
&
\mathop\mathrm{tr}\bm{F}_{3}^{2}
&=
4,
\\
\bm{F}_{4}
&=
\left(
\begin{array}{cc}
0 & 2
\\
4 & 0
\end{array}
\right),
&
\bm{F}_{4}^{2}
&=
\left(
\begin{array}{cc}
6 & 2
\\
4 & 6
\end{array}
\right),
&
\mathop\mathrm{tr}\bm{F}_{4}^{2}
&=
6,
\\
\bm{F}_{5}
&=
\left(
\begin{array}{cr}
0 & -2
\\
4 & 0
\end{array}
\right),
&
\bm{F}_{5}^{2}
&=
\left(
\begin{array}{cr}
2 & -2
\\
4 & 2
\end{array}
\right),
&
\mathop\mathrm{tr}\bm{F}_{5}^{2}
&=
2,
\end{align*}
and then conclude that the matrices $\bm{G}_{3}$, $\bm{G}_{4}$ and $\bm{G}_{5}$ are to be rejected as well.

Furthermore, for the matrix $\bm{G}_{6}$ we have 
\begin{equation*}
\bm{F}_{6}
=
\left(
\begin{array}{cr}
0 & -3
\\
1 & 0
\end{array}
\right),
\qquad
\bm{F}_{6}^{k}
=
\bm{F}_{6},
\quad
k\geq1,
\qquad
\mathop\mathrm{tr}\bm{F}_{6}
=
0.
\end{equation*}
 
Taking into account that $\mathop\mathrm{tr}\bm{H}_{6}^{7}=\mathop\mathrm{tr}\bm{F}_{6}^{7}=0$, we accept the matrix $\bm{G}_{6}$, which yields the generating matrix
\begin{equation*}
\bm{S}_{6}
=
\bm{I}\oplus\bm{H}_{6}^{6}
=
\bm{I}
\oplus
\bm{G}_{6}^{-}\bm{F}_{6}^{6}(\bm{A}\oplus\bm{B})
=
\left(
\begin{array}{rrcrrrc}
0 & 2 & 3 & -2 & 2 & 1 & 4
\\
\mathbb{0} & 0 & \mathbb{0} & \mathbb{0} & \mathbb{0} & \mathbb{0} & \mathbb{0}
\\
-5 & -3 & 0 & -5 & -1 & -3 & 1
\\
\mathbb{0} & \mathbb{0} & \mathbb{0} & 0 & \mathbb{0} & \mathbb{0} & \mathbb{0}
\\
\mathbb{0} & \mathbb{0} & \mathbb{0} & \mathbb{0} & 0 & \mathbb{0} & \mathbb{0}
\\
\mathbb{0} & \mathbb{0} & \mathbb{0} & \mathbb{0} & \mathbb{0} & 0 & \mathbb{0}
\\
\mathbb{0} & \mathbb{0} & \mathbb{0} & \mathbb{0} & \mathbb{0} & \mathbb{0} & 0
\end{array}
\right).
\end{equation*} 

The matrices $\bm{G}_{7}$ and $\bm{G}_{8}$ are rejected because
\begin{align*}
\bm{F}_{7}
&=
\left(
\begin{array}{cc}
0 & 0
\\
1 & 0
\end{array}
\right),
&
\bm{F}_{7}^{2}
&=
\left(
\begin{array}{cc}
1 & 0
\\
1 & 1
\end{array}
\right),
&
\mathop\mathrm{tr}\bm{F}_{7}^{2}
&=
1,
\\
\bm{F}_{8}
&=
\left(
\begin{array}{cr}
0 & 2
\\
1 & 0
\end{array}
\right),
&
\bm{F}_{8}^{2}
&=
\left(
\begin{array}{cc}
3 & 2
\\
1 & 3
\end{array}
\right),
&
\mathop\mathrm{tr}\bm{F}_{8}^{2}
&=
3.
\end{align*}

Finally, to decide on the matrix $\bm{G}_{9}$, we obtain 
\begin{equation*}
\bm{F}_{9}
=
\left(
\begin{array}{cr}
0 & -2
\\
1 & 0
\end{array}
\right),
\qquad
\bm{F}_{9}^{k}
=
\bm{F}_{9},
\quad
k\geq1,
\qquad
\mathop\mathrm{tr}\bm{H}_{9}
=
0.
\end{equation*}

As $\mathop\mathrm{tr}\bm{H}_{9}^{7}=0=\mathbb{1}$, we take $\bm{G}_{9}$ to produce the generating matrix
\begin{equation*}
\bm{S}_{9}
=
\bm{I}\oplus\bm{H}_{9}^{6}
=
\bm{I}
\oplus
\bm{G}_{9}^{-}\bm{F}_{9}^{6}(\bm{A}\oplus\bm{B})
=
\left(
\begin{array}{rrcrrrc}
0 & \mathbb{0} & \mathbb{0} & \mathbb{0} & \mathbb{0} & \mathbb{0} & \mathbb{0}
\\
\mathbb{0} & 0 & \mathbb{0} & \mathbb{0} & \mathbb{0} & \mathbb{0} & \mathbb{0}
\\
-4 & -2 & 0 & -5 & -1 & -3 & 1
\\
\mathbb{0} & \mathbb{0} & \mathbb{0} & 0 & \mathbb{0} & \mathbb{0} & \mathbb{0}
\\
\mathbb{0} & \mathbb{0} & \mathbb{0} & \mathbb{0} & 0 & \mathbb{0} & \mathbb{0}
\\
-1 & 1 & 2 & -3 & 1 & 0 & 3
\\
\mathbb{0} & \mathbb{0} & \mathbb{0} & \mathbb{0} & \mathbb{0} & \mathbb{0} & 0
\end{array}
\right).
\end{equation*} 

Now, it remains to combine the columns of the matrices $\bm{S}_{1}$, $\bm{S}_{6}$ and $\bm{S}_{9}$ into one set of columns, and then refine this set by removing those columns which are linearly dependent on others. First, we take the columns from $\bm{S}_{1}$ and note that they constitute a linearly independent system since each column starting from the second has a nonzero element which is the only nonzero element in its row.

Next, we find columns in $\bm{S}_{6}$ that are independent of columns in $\bm{S}_{1}$, and thus have to be included in the set of generators. To apply the criterion of Lemma~\ref{L-AbAb1} to all columns of $\bm{S}_{6}$ simultaneously, we evaluate the matrix  
\begin{equation*}
(\bm{S}_{1}(\bm{S}_{6}^{-}\bm{S}_{1})^{-})^{-}\bm{S}_{6}
=
\left(
\begin{array}{rrcrcrc}
0 & 2 & 5 & 0 & 4 & 2 & 6
\\
-2 & 0 & 3 & -2 & 2 & 0 & 4
\\
-3 & -1 & 2 & -3 & 1 & -1 & 3
\\
2 & 4 & 7 & 2 & 6 & 4 & 8
\\
-2 & 0 & 3 & -2 & 2 & 0 & 4
\\
-1 & 1 & 4 & -1 & 3 & 1 & 5
\\
-4 & -2 & 1 & -4 & 0 & -2 & 2
\end{array}
\right),
\end{equation*}
and inspect its diagonal for the entries equal to $0=\mathbb{1}$.

Since the first two diagonal elements are zeros, we drop the corresponding two columns of $\bm{S}_{6}$. Adding the other columns to the matrix $\bm{S}_{1}$ yields the matrix
\begin{equation*}
\bm{S}_{1,6}
=
\left(
\begin{array}{ccccccccrrrc}
0 & 2 & 5 & 0 & 4 & 2 & 6 & 3 & -2 & 2 & 1 & 4
\\
\mathbb{0} & 0 & \mathbb{0} & \mathbb{0} & \mathbb{0} & \mathbb{0} & \mathbb{0} & \mathbb{0} & \mathbb{0} & \mathbb{0} & \mathbb{0} & \mathbb{0}
\\
\mathbb{0} & \mathbb{0} & 0 & \mathbb{0} & \mathbb{0} & \mathbb{0} & \mathbb{0} & 0 & -5 & -1 & -3 & 1
\\
\mathbb{0} & \mathbb{0} & \mathbb{0} & 0 & \mathbb{0} & \mathbb{0} & \mathbb{0} & \mathbb{0} & 0 & \mathbb{0} & \mathbb{0} & \mathbb{0}
\\
\mathbb{0} & \mathbb{0} & \mathbb{0} & \mathbb{0} & 0 & \mathbb{0} & \mathbb{0} & \mathbb{0} & \mathbb{0} & 0 & \mathbb{0} & \mathbb{0}
\\
\mathbb{0} & \mathbb{0} & \mathbb{0} & \mathbb{0} & \mathbb{0} & 0 & \mathbb{0} & \mathbb{0} & \mathbb{0} & \mathbb{0} & 0 & \mathbb{0}
\\
\mathbb{0} & \mathbb{0} & \mathbb{0} & \mathbb{0} & \mathbb{0} & \mathbb{0} & 0 & \mathbb{0} & \mathbb{0} & \mathbb{0} & \mathbb{0} & 0
\end{array}
\right).
\end{equation*}

We examine the columns of the matrix $\bm{S}_{9}$ and calculate the matrix
\begin{equation*}
(\bm{S}_{1,6}(\bm{S}_{9}^{-}\bm{S}_{1,6})^{-})^{-}\bm{S}_{9}
=
\left(
\begin{array}{rrrrrrc}
0 & 2 & 4 & -1 & 3 & 1 & 5
\\
-2 & 0 & 2 & -3 & 1 & -1 & 3
\\
-3 & -1 & 0 & -5 & -1 & -2 & 1
\\
2 & 4 & 5 & 0 & 4 & 3 & 6
\\
-2 & 0 & 1 & -4 & 0 & -1 & 2
\\
-1 & 1 & 3 & -2 & 2 & 0 & 4
\\
-4 & -2 & -1 & -6 & -2 & -3 & 0
\end{array}
\right).
\end{equation*}

Observing that all diagonal entries are equal to $0=\mathbb{1}$, all columns in the matrix $\bm{S}_{9}$ can be dropped. As a result, we set $\bm{S}=\bm{S}_{1,6}$ to describe all regular solutions of the inequality in the form
\begin{equation*}
\bm{x}
=
\bm{S}\bm{u},
\qquad
\bm{u}>\bm{0}.
\end{equation*}
\end{example}

Finally, note that the number of columns in the generating matrix $\bm{S}$ is $12$, which is less than the number of elements in the sets of generating vectors, obtained in \cite{Wagneur2009Tropical,Sergeev2011Basic}. At the same time, it is not difficult to verify using the criterion of Lemma~\ref{L-AbAb1} that the vectors in both these sets are linearly dependent on columns in the matrix $\bm{S}$, and thus cannot generate solutions other than those already produced by $\bm{S}$.

\section{Conclusion}

In this paper, we have examined a two-sided inequality (where the unknown vector multiplied by given matrices appears on both sides) in a general setting of an arbitrary idempotent semifield. There are some algorithmic techniques and computational procedures developed in the last years to solve the inequality, which, however, cannot guarantee polynomial-time complexity of the solution. This makes it rather expedient and advisable to develop new methods and techniques that could complement and supplement existing solutions.

To solve the inequality, we used an approach that transforms it into a collection of more simple inequalities with matrices obtained by sparsification of a given matrix. These inequalities are solved analytically in explicit form, which yields a complete solution of the two-sided inequality given by a family of sets, each defined in a parametric form by a generating matrix. Since in practical problems the number of sparse matrices to define the members of the family can be sufficiently large, we have proposed a backtracking procedure which discards matrices that do not produce solutions, and hence reduces the computational cost. 

Possible lines of future research include a thorough evaluation of the computational complexity of the solution and an extension of the results to solve two-sided equations.

\subsection*{Acknowledgments}

The author sincerely thanks the anonymous referee for the insightful comments, valuable suggestions and corrections, which have been incorporated in the revised paper. He is especially grateful for providing a list of important references and for giving a simple solution for Example~\ref{X-1} from scratch.

\bibliographystyle{abbrvurl}

\bibliography{Complete_solution_of_tropical_vector_inequalities_using_matrix_sparsification}

\end{document}